\documentclass[preprint,review,10pt]{amsart}
\usepackage{amsmath,amssymb,amsfonts,xspace}
\newtheorem{theorem}{Theorem}[section]

\theoremstyle{definition}
\newtheorem{definition}[theorem]{Definition}
\newtheorem{example}[theorem]{Example}

\newtheorem{corollary}[theorem]{Corollary}

\theoremstyle{remark}

\numberwithin{equation}{section}

\begin{document}

\title{      classical prime subhypermodules and related classes 
 }

\author{Mahdi  Anbarloei}
\address{Department of Mathematics, Faculty of Sciences,
Imam Khomeini International University, Qazvin, Iran.
}

\email{m.anbarloei@sci.ikiu.ac.ir}





\begin{abstract}
In this paper, we extend the notion of prime subhypermodules
to  $n$-ary classical prime, $n$-ary weakly classical prime and $n$-ary $\phi$-classical prime subhypermodules of an $(m,n)$-hypermodule over a commutative Krasner $(m,n)$-hyperring. Many properties and characterizations of them are introduced. Moreover, we investigate the behavior of these structures under hypermodule homomorphisms,  quotient hypermodules and cartesian product. We think the knowledge gained in this setting provides a significant step in the general investigation of  subhypermodules. 
\end{abstract}
\maketitle
\section{Introduction}
To extend the notion of prime  ideals from the category of rings to the category of modules has excited several researchers to show that many, but not all, of the results in the theory of rings are also valid for modules. The concept of classical prime submodules as an extension of prime submodules was introduced by Behboodi and Koohy in \cite{beh}. A proper submodule $Q$ of $M$ is said to be  a classical prime submodule, if for each $r,s \in R$ and $a \in M$, $rsm \in Q$ implies that $ra \in Q$ or $sa \in Q$. Moreover, the notion of weakly classical prime submodules, which is a generalization of classical prime submodules was studied in \cite{hoj}.

The theory of algebraic  hyperstructures playing an important role in
the classical algebraic theory was born
in 1934 by a French mathematician, F. Marty, at the $8^{th}$ Congress of Scandinavian Mathematicians. A comprehensive review of the theory of hyperstructures appears in \cite{2, 3, 4, 5}.

 The concept of $n$-ary algebras  was introduced  by Kasner in a lecture in a annual meeting in 1904 \cite{6}. The first paper on the theory of $n$-ary groups was  written by  Dorente in 1928 \cite{7}. Moreover, for the
first time in \cite{8}
the notion of Krasner hyperrings was introduced by Krasner.  Some properties on this hyperrings can be seen in \cite{9,10}. The concept of $n$-ary hypergroups was defined  in \cite{11} as an extension  of hypergroups in the sense of Marty. After the introduction of the concept of $(m,n)$-hyperrings in \cite{12}, Davvaz et al. extended $(m, n)$-rings to  Krasner $(m, n)$- hyperrings  and studied some results in this context in \cite{13}.  Several classes of hyperideals namely  maximal hyperideal, $n$-ary prime hyperideal, $n$-ary primary hyperideal and the radical of a hyperideal in a Krasner $(m,n)$-hyperring were introduced in \cite{sorc1}. \\
\cite{13} A commutative  Krasner $(m, n)$-hyperring with a scalar identity
 $1$ is an algebraic hyperstructure $(R, f^{\prime}, g^{\prime})$ if the following hold:
(1) $(R, f^{\prime}$) is a canonical $m$-ary hypergroup, 
(2) $(R, g^{\prime})$ is a commutative $n$-ary semigroup, 
(3) the $n$-ary operation $g^{\prime}$ is distributive with respect to the $m$-ary hyperoperation $f^{\prime}$, i.e.,  $g^{\prime}(a^{i-1}_1, f^{\prime}(x^m _1 ), a^n _{i+1}) = f^{\prime}(g^{\prime}(a^{i-1}_1, x_1, a^n_{ i+1}),..., g^{\prime}(a^{i-1}_1, x_m, a^n_{ i+1}))$, for each $a^{i-1}_1 , a^n_{ i+1}, x^m_ 1 \in R$, and $1 \leq i \leq n$,
(4) $0$ is a zero element  of the $n$-ary operation $g^{\prime}$, i.e., for every $x^n_ 2 \in R$ we have $g^{\prime}(0, x^n _2) = g^{\prime}(x_2, 0, x^n _3) = ... = g^{\prime}(x^n_ 2, 0) = 0$,
(5) for all $x \in R$, $g(x,1^{(n-1)})=x$.

The sequence $x_i, x_{i+1},..., x_j$ is denoted by $x^j_i$. For $j < i$, $x^j_i$ is the empty symbol. In this convention
$f^{\prime}(x_1,..., x_i, y_{i+1},..., y_j, z_{j+1},..., z_n)$
will be written as $f^{\prime}(x^i_1, y^j_{i+1}, z^n_{j+1})$. In the case when $y_{i+1} =... = y_j = y$ the last expression will be written in the form $f^{\prime}(x^i_1, y^{(j-i)}, z^n_{j+1})$. 
For non-empty subsets $A_1,..., A_n$ of $R$ we define
$f^{\prime}(A^n_1) = f^{\prime}(A_1,..., A_n) = \bigcup \{f^{\prime}(x^n_1) \ \vert \  x_i \in  A_i, i = 1,..., n \}$. A non-empty subset $S$ of $R$ is called a subhyperring of $R$ if $(S, f^{\prime}, g^{\prime})$ is a Krasner $(m, n)$-hyperring. Let
$I$ be a non-empty subset of $R$, we say that $I$ is a hyperideal of $(R, f^{\prime}, g^{\prime})$ if $(I, f^{\prime})$ is an $m$-ary subhypergroup of $(R, f^{\prime})$ and $g^{\prime}(x^{i-1}_1, I, x_{i+1}^n) \subseteq I$, for every $x^n _1 \in  R$ and  $1 \leq i \leq n$. For each element $x \in R$, the hyperideal generated by $x $ is denoted by $\langle x \rangle$ and defined as  $\langle x \rangle =g(R,x,1^{(n-2)})=\{g(r,x,1^{(n-2)}) \ \vert \ r \in R \}$.
Recall from \cite{sorc1} that a proper hyperideal $P$ of a Krasner $(m, n)$-hyperring $(R, f^{\prime}, g^{\prime})$ is an $n$-ary prime hyperideal if for hyperideals $I_1,..., I_n$ of $R$, $g^{\prime}(I_1^ n) \subseteq P$ implies that $I_1 \subseteq P$ or $I_2 \subseteq P$ or ...or $I_n \subseteq P$. Also, Lemma 4.5 in \cite{sorc1} shows that a proper hyperideal $P$  of a Krasner $(m, n)$-hyperring $(R, f^{\prime}, g^{\prime})$ is an $n$-ary prime hyperideal if for all $x^n_ 1 \in R$, $g^{\prime}(x^n_ 1) \in P$ implies that $ x_i \in P$ for some $1 \leq i \leq n$. 

 Hypermodules over a hyperring is a generalization of the classical modules over a ring. Several types
of  hypermodules were introduced  by many authors. The notion of $(m, n)$-hypermodules over $(m, n)$-hyperrings was defined in \cite{16}. After, some classes of the hypermodules  were studied in \cite{17, 18, sorc2}. Prime and  primary   subhypermodules of an $(m,n)$-hypermodule were discussed in \cite{sorc3}.

Motivated and inspired by the above papers, the purpose of this research  work is to introduce and study generalizations
of prime subhupermodules. We define the notions of classical prime, weakly classical prime and $\phi$-classical prime subhypermodules of an $(m,n)$-hypermodule over a commutative Krasner $(m,n)$-hyperring with a scalar identity  $1$. Then a number of major conclusions are given to explain the general framework of these structures. Moreover, we give some characterizations of these concepts on cartesian product of $(m,n)$-hypermodules. 
\section{Preliminaries}
In this section, we recall some basic terms and definitions  concerning $n$-ary hyperstructures which we need to develop our paper.\\
\begin{definition} (\cite{16})
Let $M$ be a nonempty set. Then $(M, f, g)$ is an $(m, n)$-hypermodule over an $(m, n)$-
hyperring $(R, f^{\prime}, g^{\prime})$, simply $R$, if $(M, f)$ is a canonical $m$-ary hypergroup and the map 

$g:\underbrace{R \times ... \times R}_{n-1} \times M\longrightarrow 
P^*(M)$\\
statisfied the following conditions:

$(i)\  g(r_1^{n-1},f(x_1^m))=f(g(r_1^{n-1}
,x_1),...,g(r_1^{n-1}
,x_m))$

$(ii)\  g(r_1^{i-1},f^{\prime}(s_1^m),r_{i+1}^{n-1},x)=f(g(r_1^{i-1}
,s_1,r_{i+1}^{n-1},x),...,g(r_1^{i-1}
s_m,r_{i+1}^{n-1},x))$

$(ii)\  g(r_1^{i-1},g^{\prime}(r_i^{i+n-1}),r_{i+m}^{n+m-2},x)=
g(r_1^{n-1},g(r_m^{n+m-2},x))$

$ (iv) \ \{0\}=g(r_1^{i-1},0,r_{i+1}^{n-1},x)$.
\end{definition} 
If $g$ is an $n$-ary hyperoperation, $A_1,...,A_{n-1}$ are subsets of $R$ and $M^{\prime} \subseteq M$, we set

$g(A_1^{n-1},M^{\prime})=\bigcup\{g(r_1^{n-1},m)\ \vert \ r_i \in A_i, 1 \leq i \leq n-1, m \in M^{\prime}\}.$\\
Let $1$ be  a scalar identity in $R$. For every $a \in M$ and $r_1^{n-1} \in R$ we have
\[g(1^{(n-1)},a)=\{a\}, \hspace{1cm}g(0^{(n-1)},a)=\{0\}, \hspace{1cm} g(r_1^{n-1},0)=\{0\}.\]

Let $(M, f, g)$ be an $(m, n)$-hypermodule over $R$. A non-empty subset $N$ of $M$ is said to be  an $(m,n)$-subhypermodule of $M$ if $(N,f)$ is a $m$-ary subhypergroup of $(M,f)$ and $g(R^{(n-1)},N) \in P^*(N)$.

\cite{sorc2} Let $(M, f, g)$ be an $(m, n)$-hypermodule,  $N$  a subhypermodule of $M$ and $a$ an element of $M$. Then the hyperideals $S_N$ and $N_a$ is considered as follows:

$S_N=\{r \in R \ \vert \ g(r,1^{(n-2)},M) \subseteq N\}$

 $N_a=\{r \in R \ \vert \ g(r,1^{(n-2)},a) \subseteq N\}$
 
\begin{definition} (\cite{sorc3})
Let $M$ be an $(m, n)$-hypermodule over  $R$. A proper subhypermodule $K$ of $M$ is said to be
maximal, if for $N \leq M$ with $K \subseteq  N \subseteq  M$, we have either $K = N$ or $N = M$. 
\end{definition}  
\begin{definition} (\cite{sorc3})
Let $M$ be an $(m, n)$-hypermodule over  $R$. A proper subhypermodule $N$ of $M$ is said to be
$n$-ary prime, if $g(r_1^{n-1},a)\subseteq N$ with $r_1^{n-1} \in R$ and $a \in M - N$, implies that $g(r_1^{n-1},M) \subseteq N$. 
\end{definition}
In \cite{sorc2}, there exists another definition of $n$-ary prime subhypermodules which is equivalent to above definition. A proper subhypermodule $N$ of $M$ is called   
$n$-ary prime, if $g(r_1^{n-1},a)\subseteq N$ with $r_1^{n-1} \in R$  implies that $a \in  N$ or $r_i \in S_N$ for some $1 \leq i \leq n-1$.



\begin{definition}(\cite{sorc3})
Let $N$ be a subhypermodule of an  $(m, n)$-hypermodule $(M, f, g)$ over  $R$. Then the set

$M/N = \{f(x^{i-1}_1, N, x^m_{i+1}) \ \vert \  x^{i-1}_1,x^m_{i+1} \in M \}$\\
endowed with $m$-ary hyperoperation $f$ which for all $x_{11}^{1m},...,x_{m1}^{mm} \in M$

$F(f(x_{11}^{1 (i-1)}, N, x^{1m}_ {1(i+1)}),..., f(x_{m1}^{ m(i-1)}, N, x^{mm}_ {m(i+1)}))$ 

$\hspace{0.3cm}= \{ (f(t_1^{i-1}, N,t_{i+1}^m ) \ \vert \ t_1 \in f(x_{11}^{m1}),...,t_m \in  f(x_{1m}^{mm})\}$\\
and with $n$-ary hyperoperation $G:\underbrace{R \times ... \times R}_{n-1} \times M/N \longrightarrow 
 P^*(M/N)$ which for all $x_1^{i-1},x_{i+1}^m\in M$ and $r_1^{n-1} \in R$

$G(r_1^{n-1},f(x_1^{i-1},N,x_{i+1}^m))$

$\hspace{0.3cm}=\{f(z_1^{i-1},N,z_{i+1}^m)\ \vert \ z_1 \in g(r_1^{n-1},x_1),..., z_m \in g(r_1^{n-1},x_m)\}$\\
is an  $(m, n)$-hypermodule over  $R$, and $(M/N, F, G)$  is called the quotient $(m, n)$-hypermodule  of $M$ by $N$. 
\end{definition} 




\begin{definition} (\cite{sorc2})
For every nonzero element $m$ of $(m, n$)-hypermodule $(M, f , g)$ over $R$, we define

$F_m
= \{r \in  R  \ \vert \ 0 \in  g(r, 1^ {(n-2)}, m);  r \neq  0\}.$\\
It is clear that $F_m$ is a hyperideal of $(R, h, k)$.  The $(m, n)$-hypermodule $(M, f , g)$ is said to be  faithful,
if $F_m = \{0\}$ for all nonzero elements $m \in  M$, that is
$0 \in  g(r, 1^{( n-2)}, m)$ implies that  $r = 0$, for $r \in  R$. 
\end{definition}
\begin{definition}
(\cite{sorc3}) Assume that $(M_1,f_1,g_1)$ and $M_2,f_2,g_2)$ are two $(m,n)$-hypermodules over $R$. A mapping $h:  M_1 \longrightarrow M_2$  is a homomorphism of $(m, n)$-hypermodules if for all $a_1^m, a \in M_1$ and $r_1^{n-1} \in R$:
\[h(f_1(a_1^m))=f_2(h(a_1),\cdots,h(a_m)),\]
$\hspace{3.7cm} h(g_1(r_1^{n-1},a))=g_2(r_1^{n-1},h(a)).$
\end{definition}

\section{$n$-ary classical prime subhypermodules}
In this section, we want to consider the concept of an $n$-ary classical prime subhypermodule which is a generalization of the concept of prime submodules.
\begin{definition}
Let $Q$ be a proper subhypermodule of an $(m,n)$-hypermodule $M$ over $R$. $Q$ refers to an $n$-ary classical prime subhypermodule if for $r_1^{n-1} \in R$ and $a \in M$, $g(r_1^{n-1},a) \subseteq Q$ implies that $g(r_i,1^{(n-2)},a) \subseteq Q$ for some $1 \leq i \leq n-1$.
\end{definition}
\begin{example}
Suppose that $R=\{0,1,2\}$ and define $3$-ary hyperoperation $f^{\prime}$ and $3$-ary operation $g^{\prime}$ on $R$ as follows:

$f^{\prime}(0,0,0)=0, \hspace{0.8cm}f^{\prime}(0,0,2)=2, \hspace{0.8cm}f^{\prime}(1,1,2)=f^{\prime}(1,2,2)=\{1,2\},  $

$f^{\prime}(0,0,1)=1, \hspace{0.8cm}f^{\prime}(2,2,2)=2, \hspace{0.8cm}f^{\prime}(0,1,1)=\{0,1\},  $

$f^{\prime}(1,1,1)=1, \hspace{0.8cm}f^{\prime}(0,1,2)=R, \hspace{0.7cm}f^{\prime}(0,2,2)=\{0,2\},  $\\
and 

$g^{\prime}(r_1,r_2,r_3)=0;  \hspace{0.8cm} r_i=0,  \hspace{0.4cm}1 \leq i \leq 3 $

$g^{\prime}(1,1,1)=1, \hspace{1.2cm}g^{\prime}(1,2,2)=g^{\prime}(1,1,2)=g^{\prime}(2,2,2)=2  $\\
Then $(R,f^{\prime},g^{\prime})$ is a commutative $(3,3)$-hyperring. Now, consider the set $M=\{0,1,2,3\}$. $(M,f,g)$   is an $(3,3)$-hypermodule over $R$ with $3$-ary hyperoperations $f$ and $3$-ary external hyperoperation $g$ defined by:

 $f(0,0,0)=0, \hspace{1.7cm}f(0,0,1)=1, \hspace{1.8cm}f(0,0,2)=2,  \hspace{0.8cm}$

 $f(0,0,3)=3,\hspace{1.7cm} f(1,1,1)=1, \hspace{1.8cm} f(2,2,2)=2,$
 
 $f(3,3,3)=3,  \hspace{1.7cm}f(0,1,1)=\{0,1\},\hspace{1.1cm} f(0,2,2)=\{0,2\},$
 
 $ f(0,3,3)=\{0,3\},\hspace{1cm}f(1,2,3)=\{1,2,3\},\hspace{0.7cm}f(0,1,2)=\{0,1,2\}, $

 $ f(0,1,3)=\{0,1,3\},\hspace{0.7cm}f(0,2,3)=\{0,2,3\},\hspace{0.7cm}f(2,2,3)=f(2,3,3)=\{2,3\},$
 
 $ f(1,1,2)=f(1,2,2)=\{1,2\},\hspace{3cm}f(1,1,3)=f(1,3,3)=\{1,3\},$\\and for $r_1^2 \in R$ and $a \in M,$
\[ 
g(r_1^2,a)=
\begin{cases}
\{0\} & \text{if $r_1=0$  or $r_2=0$ or $a=0$},\\
\{2\}  & \text{if $r_1, r_2 \neq 0$ and $a \neq 0$,}\\
\{a\} & \text{if $r_1=r_2=0$.}
\end{cases}\] 
Let $Q=\{0,2\}$. Then $Q$ is a $3$-ary classical prime subhypermodule of $M$.
\end{example}
\begin{theorem} \label{11}
Let $Q$ be a proper subhypermodule of an $(m,n)$-hypermodule $M$ over $R$. Then $Q$ is an $n$-ary classical prime subhypermodule if and only if for  hyperideals $I_1^{n-1}$ of $R$ and subhypermodule $N$ of $M$, if $g(I_1^{n-1},N) \subseteq Q$,  then $g(I_i,1^{(n-2)},N) \subseteq Q$ for some $1 \leq i \leq n-1$.
\end{theorem}
\begin{proof}
This can be proved by using an argument similar to that in the proof of Theorem 2.14 in \cite{darani}.
\end{proof}
\begin{theorem} \label{12}
Let $Q$ be a proper subhypermodule of an $(m,n)$-hypermodule $M$ over $R$ and let $S=M-Q$. Then $Q$ is an $n$-ary classical prime subhypermodule  of $M$  if and only if for  hyperideals $I_1^{n-1}$ of $R$ and for subhypermodules $N_1,N_2$ of $M$,  $f(N_1,g(I_i,1^{(n-2)},N_2),0^{(m-2)}) \cap S \neq \varnothing$ for all $1 \leq i \leq n-1$ implies that $f(N_1,g(I_1^{n-1},N_2),0^{(m-2)}) \cap S \neq \varnothing$.
\end{theorem}
\begin{proof}
$(\Longrightarrow)$ Let $I_1^{n-1}$ be hyperideals of $R$ and let $N_1$ and $N_2$ be two subhypermodules of an $(m,n)$-hypermodule $M$ over $R$ with $f(N_1,g(I_i,1^{(n-2)},N_2),0^{(m-2)}) \cap S \neq \varnothing$ for all $1 \leq i \leq n-1$. Suppose that $f(N_1,g(I_1^{n-1},N_2),0^{(m-2)}) \cap S = \varnothing$. This implies $g(I_1^{i-1},N_2) \subseteq Q$. Then we get $g(I_i,1^{(n-2)},N_2) \subseteq Q$ for some $1 \leq i \leq n-1$ since $Q$ is an $n$-ary classical prime subhypermodule of $M$. Thus we obtain $f(N_1,g(I_i,1^{(n-2)},N_2),0^{(m-2)}) \cap S = \varnothing$ which is a contradiction. \\
$(\Longleftarrow)$ Let  $g(I_1^{n-1},N) \subseteq Q$ for  hyperideals $I_1^{n-1}$ of $R$ and for a subhypermodule $N$ of an $(m,n)$-hypermodule $M$ over $R$ but $g(I_i,1^{(n-2)},N) \nsubseteq Q$ for all $1 \leq i \leq n-1$. Then we conclude that $g(I_i,1^{(n-2)},N) \cap S \neq \varnothing$ for all $1 \leq i \leq n-1$ which means $g(I_1^{n-1},N) \cap S \neq \varnothing$ which is a contradiction. Thus $Q$ is an $n$-ary classical prime subhypermodule of $M$.
\end{proof}
\begin{theorem}
Let $Q$ be a proper subhypermodule of an $(m,n)$-hypermodule $M$ over $R$. Let $S$ be a nonempty subset of $ M-\{0\}$  such that for  hyperideals $I_1^{n-1}$ of $R$ and for subhypermodules $N_1,N_2$ of $M$,  $f(N_1,g(I_i,1^{(n-2)},N_2),0^{(m-2)}) \cap S \neq \varnothing$ for all $1 \leq i \leq n-1$ implies that $f(N_1,g(I_1^{n-1},N_2),0^{(m-2)}) \cap S \neq \varnothing$. If  $Q$ is maximal with respect to the property that $Q \cap S= \varnothing$, then $Q$ is an $n$-ary classical prime subhypermodule  of $M$.
\end{theorem}
\begin{proof}
Assume that $g(I_1^{n-1},N) \subseteq Q$ for some hyperideals $I_1^{n-1}$ of $R$ and for a subhypermodule $N$ of $M$. Let $g(I_i,1^{(n-2)},N) \nsubseteq Q$ for all $1 \leq i \leq n - 1$. Then $f(Q,g(I_i,1^{(n-2)},N),0^{(m-2)}) \cap S \neq \varnothing$ for all $1 \leq i \leq n-1$ by the maximality of $Q$. This implies that $f(Q,g(I_1^{n-1},N),0^{(m-2)}) \cap S \neq \varnothing$ which means $Q \cap S \neq \varnothing$ which is a contradiction. Consequently, $Q$ is an $n$-ary classical prime subhypermodule  of $M$.
\end{proof}
Recall from \cite{18} that if $N$ is a subhypermodule of $(M,f,g)$ over $R$, then we consider the set $M/N$ as follows:
\[M/N=\{f(a,N,0^{(m-2)}) \ \vert \ a \in M \}.\]
Moreover, recall from \cite{18} that an element $a$ of  an $(m,n)$-hypermodule $M$ over $R$ is called torsion free if $g(r_1^{n-1},a)=0$, then there exists $1 \leq i \leq n-1 $ such that $r_i=0$. If all 
elements of $M$ are torsion free, then $M$ is called torsion free.
\begin{theorem}
Suppose that  $M$ is  an $(m,n)$-hypermodule  over $R$ such that every classical prime
subhypermodule of $M$ is an intersection of maximal subhypermodules of $M$ and $N$ is a subhypermodule of $M$. If $M/N$ is  a torsion free  $(m,n)$-hypermodule  over $R$, then every classical prime subhypermodule of $N$ is an intersection of maximal subhypermodules of $N$.
\end{theorem}
\begin{proof}
Assume that $Q$ is a classical prime subhypermodule of $N$. Let $g(r_1^{n-1},m) \subseteq Q$ for some $r_1^{n-1} \in R$ and $m \in M$. If $m \in N$, then  $Q$ is a classical prime subhypermodule of $M$. So  suppose that $m \notin N$. Then we have $g(r_1^{n-1},m) \subseteq Q \subseteq N$. Since $m \notin N$ and $M/N$ is  a torsion free  $(m,n)$-hypermodule  over $R$, we obtain $r_i=0$ for some $1 \leq i \leq n-1$. Therefore we get $g(r_i,1^{(n-2)},m) \subseteq Q$. This means that $Q$ is a classical prime subhypermodule of $M$. By the hypothesis, we infer that $Q$ is an intersection of maximal subhypermodules of $M$. Put $Q=\cap_{i \in I}K_i$ for the  maximal subhyperideals $K_i$ of $M$. Consider $Q_i=K_i \cap N$ for each $i \in I$. Clearly $Q=\cap_{i \in I}Q_i$, because $Q \subseteq N$. We assume that $Q_i \subset N$ for every $i \in I$. Let $x \in N-Q_i$ for some $i \in I$. This means $x \notin K_i$. By maximality $K_i$ of $M$, we conclude that $f(K_i,\langle x \rangle,0^{(m-2)})=M$. Assume that $a \in N$. Then there exists some $a_i \in K_i$ and $r_i^{n-1} \in R$ such that $a \in f(a_i,g(r_1^{n-1},x),0^{(m-2)})$. Thus we have $a_i \in f(a,-g(r_1^{n-1},x),0^{(m-2)}) \subseteq N$ which implies $a_i \in Q_i$. So $a \in f(a_i,\langle x \rangle,0^{(m-2)}) \in f(Q_i,\langle x \rangle,0^{(m-2)})$ which means $f(Q_i,\langle x \rangle,0^{(m-2)})=N$. Hence $Q_i$ is a maximal subhypermodule of $N$, as needed.
\end{proof}
\section{$n$-ary weakly classical prime subhypermodules}
\begin{definition}
Let $Q$ be a proper subhypermodule of   an $(m,n)$-hypermodule $M$ over $R$. $Q$ is called an  $n$-ary weakly classical prime subhypermodule if  $0 \notin g(r_1^{n-1},a) \subseteq Q$ for $r_1^{n-1} \in R$ and $a \in M$, then $g(r_i,1^{(n-2)},a) \subseteq Q$ for some $1 \leq i \leq n-1$.
\end{definition}
\begin{example}
Consider the commutative group $(H=\{0,x,y,z\},\oplus)$, where $\oplus$ is defined by 

\hspace{3cm}
\begin{tabular}{c|c} 
$\oplus$ & $0$ \ \ \ \ \  \ \ $x$ \ \ \ \ \ \ \ $y$ \ \ \ \ \  \ $z$
\\ \hline 
0 & $0$\ \ \ \ \  \ \ $x$ \ \ \ \ \ \ \ $y$\ \ \ \ \ \ \  $z$ 
\\ $x$ & $x$\ \ \ \ \  \ \  $0$  \ \ \ \ \ \ \ $z$ \ \ \ \ \ \ \ $y$
\\ $y$ & $y$\ \ \ \ \  \ \  $z$  \ \ \ \ \ \ \ $0$ \ \ \ \ \ \ \  $x$
\\ $z$ & $z$\ \ \ \ \  \ \  $y$  \ \ \ \ \ \ \ $x$ \ \ \ \ \ \ \  $0$
\end{tabular}\\ 
It is clear that $H$ is a $\mathbb{Z}$-module. Also, the ring of integers $\mathbb{Z}$ is a Krasner $(3,3)$-hyperring with $3$-ary hyperoperation $f^{\prime}(r_1^3)=\{r_1+r_2+r_3\}$ and $3$-ary operation $g^{\prime}(r_1^3)=r_1 \cdot r_2 \cdot r_3$ for all $r_1^3 \in \mathbb{Z}$. Now, we have the canonical $(3,3)$-hypermodule $(H,f,g)$ over $(\mathcal{Z},f^{\prime},g^{\prime})$ where $3$-ary hyperoperation $f$ and $3$-ary external hyperoperation $g$ on $H$ are defined as follows:

$f(a,a,a)=\{a\}, \hspace{0.3cm}$ for  $a \in H$

$f(0,a,a)=\{0\}, \hspace{0.3cm}$ for  $a \in H$

$f(a,a,b)=\{b\},  \hspace{0.3cm}$ for  $a, b \in H$

$f(a,b,c)=\{d\}, \hspace{0.3cm}$ for  $a \neq b \neq c \neq d \in H $\\
and 

$g(r_1^2,a)=\{\underbrace{a \oplus \cdots \oplus a}_{r_1 \cdot r_2}\}, \hspace{0.3cm}$   for $r_1^2 \in \mathbb{Z}$ and $a \in H$.\\ The subhypermodule $Q=\{0,y\}$ is a $3$-ary weakly classical prime subhypermodule of $H$.
\end{example}
\begin{theorem} \label{21}
Let $Q$ be an $n$-ary weakly classical prime subhypermodule of an $(m,n)$-hypermodule $M$ over $R$ and $a \in M-Q$ such that $F_a=\{0\}$. If 
$0 \neq g^{\prime}(r_1^n) \in Q_a$ for some $r_1^n \in R$, then $r_i \in Q_a$ for some $1 \leq i \leq n$.
\end{theorem}
\begin{proof}
Assume that $Q$ is an $n$-ary weakly classical prime subhypermodule of an $(m,n)$-hypermodule $M$ over $R$ and $a \in M-Q$ such that $F_a=0$. Suppose that  $0 \neq g^{\prime}(r_1^n) \in Q_a$ for some $r_1^n \in R$ such that $r_2^n \notin Q_a$. We must show that $r _1 \in Q_a$. By $r_2^n \notin Q_a$ we conclude that $g(r_i,1^{(n-2)},a) \nsubseteq Q$ for all $2 \leq i \leq n$. From $g^{\prime}(r_1^n) \in Q_a$ it follows that  $0 \notin g(g^{\prime}(r_1^n),1^{(n-2)},a) \subseteq Q$ because $F_a=\{0\}$. This means $0 \notin g(g^{\prime}(r_1^{n-2},g^{\prime}(r_{n-1}^n,1^{(n-2)})),1^{(n-2)},a)=g(1^{(n-1)},g(r_1^{n-2},g^{\prime}(r_{n-1}^n,1^{(n-2)}),a)=g(r_1^{n-2},g^{\prime}(r_{n-1}^n,1^{(n-2)}),a) \subseteq Q$. Since $Q$ is an $n$-ary weakly classical prime subhypermodule of $M$, we get $g(r_i,1^{(n-2)},a) \subseteq Q$ for some $1 \leq i \leq n-2$ or $g(g^{\prime}(r_{n-1}^n,1^{(n-2)}),1^{(n-2)},a)=g(r_{n-1}^n,1^{(n-3)},a) \subseteq Q$. In the second possibilty, we obtain $r_{n-1} \in Q$ or $r_n \in Q$ as $1 \notin Q$ and the proof is completed. 
\end{proof}
\begin{theorem} \label{plo}
Let $P$ and $Q$ be two  subhypermodules of an $(m,n)$-hypermodule $M$ over $R$ such that $P \subset Q$. If $P$ is an $n$-ary weakly classical prime subhypermodule of $M$ and  $Q/P$ is an $n$-ary weakly classical prime subhypermodule of $M/P$, then $Q$ is an $n$-ary weakly classical prime subhypermodule of $M$.
\end{theorem}
\begin{proof}
Assume that $0 \notin g(r_1^{n-1},a) \subseteq Q$ for $r_1^{n-1} \in R$ and $a \in M$. If $g(r_1^{n-1},a) \subseteq P$, then we are done. Suppose that $g(r_1^{n-1},a) \nsubseteq P$. So $0 \neq G(r_1^{n-1},f(a,P,0^{(m-2)}))=\{f(g(r_1^{n-1},a),P,0^{(m-2)}\} \subseteq Q/P$. Since $Q/P$ is an $n$-ary weakly classical prime subhypermodule of $M/P$, then we conclude that $G(r_i,1^{(n-2)},f(a,P,0^{(m-2)})) =\{f(g(r_i,1^{(n-2)},a),P,0^{(m-2)})\} \subseteq Q/P$ for some $1 \leq i \leq n-1$ which implies $g(r_i,1^{(n-2)},a) \subseteq Q$, as needed.
\end{proof}
Next, we observe that weakly classical prime subhypermodules behave naturally under a homomorphism.
\begin{theorem} \label{22}
Let $(M_1,f_1,g_1)$ and $(M_2,f_2,g_2)$ be two $(m,n)$-hypermodules over $(R,f^{\prime},g^{\prime})$ and let $Q_1,Q_2$ be $n$-ary  weakly classical prime subhypermodules of $M_1, M_2$, respectively. If $h:M_1 \longrightarrow M_2$ is a homomorphism,
then:
\begin{itemize} 
\item[\rm(1)]~If $h$ is an epimorphism and $Ker (h) \subseteq Q_1$, then $h(Q_1)$ is an $n$-ary  weakly classical prime subhypermodule of $M_2$.
\item[\rm(2)]~ If $h$ is a monomorphism with $h^{-1}(Q_2) \neq M_1$, then $h^{-1}(Q_2)$ is an $n$-ary  weakly classical prime subhypermodule of $M_1$.
\end{itemize} 
\end{theorem}
\begin{proof}
$(1)$  Let $0 \notin g_2(r_1^{n-1},a_2) \subseteq h(Q_1)$ for $r_1^{n-1} \in R$ and $a_2 \in M_2$. Since $h$ is an epimorphism, then there exists $a_1 \in M_1$ such that $h(a_1)=a_2$. Hence we get 
\[h(g_1(r_1^{n-1},a_1))=g_2(r_1^{n-1},h(a_1))=g_2(r_1^{n-1},a_2) \subseteq h(Q_1) \]
which means $g_1(r_1^{n-1},a_1) \subseteq Q_1$.  
Since  $Q_1$ is an $n$-ary  weakly classical prime subhypermodules of $M_1$ and $0 \notin g_1(r_1^{n-1},a_1)$, it follows that $g_1(r_i,1^{(n-2)},a_1) \subseteq Q_1$ for some $1 \leq i \leq n-1$. Therefore
\[g_2(r_i,1^{(n-2)},a_2)=g_2(r_i,1^{(n-2)},h(a_1))=h(g_1(r_i,1^{(n-2)},a_1)) \subseteq h(Q_1).\]
Thus $h(Q_1)$ is an $n$-ary  weakly classical prime subhypermodule of $M_2$.\\
$(2)$ Let $Q_2$ be an $n$-ary  weakly classical prime subhypermodules of $M_2$. Let $0 \notin g_1(r_1^{n-1},a_1) \subseteq h^{-1}(Q_2)$ for $r_1^{n-1} \in R$ and $a_2 \in M_2$. Since $h$ is a monomorphism, we conclude that $0 \notin h(g_1(r_1^{n-1},a_1)=g_2(r_1^{n-1},h(a_1)) \subseteq Q_2$.  Since $Q_2$ is an $n$-ary  weakly classical prime subhypermodules of $M_2$, we have $g_2(r_i,1^{(n-2)},h(a_1)) \subseteq Q_2$ for some $1 \leq i \leq n-1$ and so $h(g_1(r_i,1^{(n-2)},a_1)) \subseteq Q_2$. Hence $g_1(r_i,1^{(n-2)},a_1) \subseteq h^{-1}(Q_2)$ for some $1 \leq i \leq n-1$. Therefore $h^{-1}(Q_2)$ is an $n$-ary  weakly classical prime subhypermodule of $M_1$.
\end{proof}
As an immediate consequence of the previous theorem, we have the following  result.
\begin{corollary} \label{23}
Let $P$ and $Q$ be two  subhypermodules of an $(m,n)$-hypermodule $M$ over $R$ such that $P \subset Q$. If $Q$ is an $n$-ary weakly classical prime subhypermodule of $M$, then $Q/P$ is an $n$-ary weakly classical prime subhypermodule of $M/P$.
\end{corollary}
\begin{proof}
Consider the mapping $\pi:M \longrightarrow M/P$  defined by $a \longrightarrow f(a,P,0^{(n-2)})$. Then $\pi$ is an epimorphism by Theorem 3.2 in \cite{sorc3}. Suppose that $Q$ is an $n$-ary weakly classical prime subhypermodule of $M$. Since $Ker (\pi)=P \subset Q$ and $\pi$ is onto, we conclude that $\pi(Q)=Q/P$ is an $n$-ary weakly classical prime subhypermodule of $M/P$ by Theorem \ref{22} (1).
\end{proof}

Assume that $Q$ is an $n$-ary weakly classical prime subhypermodule of an $(m,n)$-hypermodule $M$ over $R$. Then $(r_1^{n-1},X)$ for $r_1^{n-1} \in R$ and some non empty subset $X$ of $M$ is called a classical $(m,n)$-zero of $Q$  if $ 0 \in g(r_1^{n-1},X) \subseteq Q$ and $g(r_i,1,X) \nsubseteq Q$ for all $1 \leq i \leq n-1$.

\begin{theorem} \label{24}
 Let $Q$ be an $3$-ary weakly classical prime subhypermodule of an $(3,3)$-hypermodule $M$ over $R$ and let $g(r_1^2,P) \subseteq Q$ for some subhypermodule $P$ of $M$ and $r_1^2 \in R$. If $(r_1^2,X)$ is not a classical $(3,3)$-zero of $Q$ for every non empty subset $X$ of $P$, then   $g(r_i,1,P) \subseteq Q$ for some  $i \in \{1,2\}$.
\end{theorem}
\begin{proof}
 Let $g(r_1^2,P) \subseteq Q$ but $g(r_i,1,P) \nsubseteq Q$ for each $i \in \{1,2\}$. This implies that for each $i \in \{1,2\}$ there exists $p_i \in P$ such that $g(r_i,1,p_i) \nsubseteq Q$. If $0 \notin g(r_1^2,p_1) \subseteq Q$, then $g(r_2,1,p_1) \subseteq Q$ since $Q$ is an $n$-ary weakly classical prime subhypermodule of $M$ and $g(r_1,1,p_1) \nsubseteq Q$. If $0 \in g(r_1^2,p_1) \subseteq Q$, then $g(r_2,1,p_1) \subseteq Q$ since $g(r_1^2,p_1)$ is not a classical $(3,3)$-zero of $Q$. Similarly, we can conclude that $g(r_1,1,p_2) \subseteq Q$. Therefore we have $g(r_1^2,f(p_1^2,0)) \subseteq Q$. This implies that $g(r_i,1,f(p_1^2,0)) \subseteq Q$ for some $i \in \{1,2\}$ which means $f(g(r_i,1,p_1),g(r_i,1,p_2),0) \subseteq Q$. If $i=1$, then we get $g(r_1,1,p_1) \subseteq Q$ which is a contradiction. If $i=2$, then we obtain $g(r_2,1,p_2) \subseteq Q$, a contradiction. Hence $g(r_i,1,P) \subseteq Q$ for some  $i \in \{1,2\}$.
\end{proof}
Suppose that $Q$ is an $n$-ary weakly classical prime subhypermodule of an $(m,n)$-hypermodule $M$ over $R$.  Let $g(I_1^{n-1},P) \subseteq Q$ for some hyperideals $I_1^{n-1}$ of $R$ and some  subhypermodule $P$ of $M$. $Q$ is called a free classical $(m,n)$-zero with respect to $g(I_1^{n-1},P)$  if $g(r_1^{n-1},X)$ is not classical $(m,n)$-zero of $Q$ for every $r_i \in I_i$ and  for every non empty subset $X$ of $P$.

\begin{corollary} \label{25}
Let $Q$ be an $3$-ary weakly classical prime subhypermodule of an $(3,3)$-hypermodule $M$ over $R$ and let $g(I_1^2,P) \subseteq Q$ for some hyperideals $I_1^2$ of $R$ and some subhypermodule $P$ of $M$. If $Q$ is a free classical $(3,3)$-zero with respect to $g(I_1^2,P)$, then $g(I_i,1,P) \subseteq Q$ for some $i \in \{1,2\}$. 
\end{corollary}
\begin{proof}
Let $g(I_i,1,P) \nsubseteq Q$ for each $i \in \{1,2\}$. Then there exists $r_i \in I_i$ for each $i \in \{1,2\}$ such that $g(r_i,1,P) \nsubseteq Q$. So we have $g(r_1^2,P) \subseteq Q$. By Theorem \ref{24}, we get $g(r_i,1,P) \subseteq Q$ for some $i \in \{1,2\}$ since $Q$ is a free classical $(3,3)$-zero with respect to $g(I_1^2,P)$. This is a contradiction. Thus $g(I_i,1,P) \subseteq Q$ for some $i \in \{1,2\}$. 
\end{proof}
\begin{theorem}
Let $Q$ be an $n$-ary weakly classical prime subhypermodule of an $(m,n)$-hypermodule $M$ over $R$. Then $Q_{g(r_1^{n-1},a)} \subseteq F_{g(r_1^{n-1},a)} \cup Q_{g(r_1,1^{(n-2)},a)} \cup \cdots \cup Q_{g(r_{n-1},1^{(n-2)},a)}$ for all $r_1^{n-1} \in R$ and $a \in M$.
\end{theorem}
\begin{proof}
Suppose that $a \in M$ and $r_1^{n-1} \in R$. 
Assume that $x \in Q_{g(r_1^{n-1},a)}$. This means that $g(x,1^{(n-2)},g(r_1^{n-1},a)) \subseteq Q$. If $0 \in g(x,1^{(n-2)},g(r_1^{n-1},a))$, then $x \in F_{g(r_1^{n-1},a)}$. If $0 \notin g(x,1^{(n-2)},g(r_1^{n-1},a))=g(r_1^{n-1},g(x,1^{(n-2)},a))$, then we conclude that $g(r_i,1^{(n-2)},g(x,1^{(n-2)},a)=g(x,1^{(n-2)},g(r_i,1^{(n-2)},a)) \subseteq Q$ for some $1 \leq i \leq n-1$ since $Q$ is an $n$-ary weakly classical prime subhypermodule of  $M$. This implies that $x \in Q_{g(r_i,1^{(n-2)},a)}$ which means $Q_{g(r_1^{n-1},a)} \subseteq F_{g(r_1^{n-1},a)} \cup Q_{g(r_1,1^{(n-2)},a)} \cup \cdots \cup Q_{g(r_{n-1},1^{(n-2)},a)}$ and the proof is completed.
\end{proof}
Recall from \cite{sorc2} that if $(M_1,f_1,g_1)$ and $(M_2,f_2,g_2)$ are two $(m,n)$-hypermodules over $R$, then the $(m,n)$-hypermodule $(M_1 \times M_2, f_1 \times f_2, g_1 \times g_2)$ over $R$ is defined by $m$-ary hyperoperation $f_1 \times f_2$ and $n$-ary external hyperoperation $g_1 \times g_2$, as follows:

$ \hspace{0.5cm} f_1 \times f_2 ((a_1,b_1), \cdots,(a_m,b_m))=\{(x_1,x_2) \ \vert \ x_1 \in f_1(a_1^m), x_2 \in f_2(b_1^m)\}$

$\hspace{0.5cm}g_1 \times g_2(r_1^{n-1},(a,b))=\{(y_1,y_2) \ \vert \ y_1 \in g_1(r_1^{n-1},a), y_2 \in g_2(r_1^{n-1},b)\}$
\vspace{0.5mm}
\begin{theorem}
Let $(M_1,f_1,g_1)$ and $(M_2,f_2,g_2)$ be two $(m,n)$-hypermodules over $R$ and $Q_1$ be a proper subhypermodule of $M_1$. Then $Q_1 \times M_2$ is an $n$-ary weakly classical prime subhypermodule of $M_1 \times M_2$ if and only if $Q_1$ is an $n$-ary weakly classical prime subhypermodule of $M_1$ and  $0 \in g_1(r_1^{n-1},a_1)$ for  $r_1^{n-1} \in R$, $a_1 \in M_1$ such that $g(r_i,1^{(n-2)},a_1) \nsubseteq Q_1$  for all $1 \leq i \leq n-1$ imply that $ g^{\prime}(r_1^{n-1},1) \in F_{a_2} $ for all $a_2 \in M_2$.
\end{theorem}
\begin{proof}
$(\Longrightarrow) $ Let $Q_1 \times M_2$ be an $n$-ary weakly classical prime subhypermodule of $M_1 \times M_2$. Suppose that $0 \notin g_1(r_1^{n-1},a_1) \subseteq Q_1$ for some $r_1^{n-1} \in R$ and for some $a_1 \in M_1$. Then we have $(0,0) \notin g_1 \times g_2(r_1^{n-1},(a_1,0)) \subseteq Q_1 \times M_2$. Therefore $g_1 \times g_2 (r_i,1^{(n-1)},(a_1,0))=\{(y_1,y_2) \ \vert \ y_1 \in g_1(r_i,1^{(n-2)},a_1), y_2 \in g_2(r_i,1^{(n-2)},0)\} \subseteq Q_1 \times M_2$ for some $1 \leq i \leq n-1$ which means $g_1(r_i,1^{(n-2)},a_1) \subseteq Q_1$. Thus $Q_1$ is an $n$-ary weakly classical prime subhypermodule of $M_1$. Suppose that $0 \in g_1(r_1^{n-1},a_1)$ for  $r_1^{n-1} \in R$, $a_1 \in M_1$ with $g(r_i,1^{(n-2)},a_1) \nsubseteq Q_1$  for all $1 \leq i \leq n-1$. Assume on the contrary that $ g^{\prime}(r_1^{n-1},1) \notin F_{a_2} $ for some $a_2 \in M_2$. This implies that $0 \notin g_2(g^{\prime}(r_1^{n-1},1),1^{(n-2)},a_2)$. It follows that $(0,0) \notin g_1 \times g_2(r_1^{n-1},(a_1,a_2)) \subseteq Q_1 \times M_2$. Since $Q_1 \times M_2$ is an $n$-ary weakly classical prime subhypermodule of $M_1 \times M_2$, we obtain $g_1 \times g_2(r_i,1^{(n-2)},(a_1,a_2))=\{(y_1,y_2) \vert \ y_1 \in g_1(r_i,1^{(n-2)},a_1), y_2 \in g_2(r_i,1^{(n-2)},a_2) \} \subseteq Q_1 \times M_2$ which implies $g_1(r_i,1^{(n-2)},a_1) \subseteq Q_1$, a contradiction. Hence $ g^{\prime}(r_1^{n-1},1) \in F_{a_2} $ for all $a_2 \in M_2$.

$(\Longleftarrow)$ Let $(0,0) \notin g_1 \times g_2(r_1^{n-1},(a_1,a_2))=\{(y_1,y_2) \vert \ y_1 \in g_1(r_1^{n-1},a_1), y_2 \in g_2(r_1^{n-1},a_2)\} \subseteq Q_1 \times M_2$ for some $r_1^{n-1} \in R$ and $(a_1,a_2) \in Q_1 \times M_2$. If $0 \notin g_1(r_1^{n-1},a_1)$, then we get $g_1(r_i,1^{(n-2)},a_1) \subseteq Q_1$ for some $1 \leq i \leq n-1$ which implies $g_1 \times g_2(r_i,1^{(n-2)},(a_1,a_2)) \subseteq Q_1 \times M_2$ for some $1 \leq i \leq n-1$, as needed. If $0 \in g_1(r_1^{n-1},a_1)$, we get $0 \notin g_2(r_1^{n-1},a_2)$ which means $g^{\prime}(r_1^{n-1},1) \notin F_{a_2}$. Then we conclude that $g_1(r_i,1^{(n-2)},a_1) \subseteq Q_1$ for some $1 \leq i \leq n-1$ which implies $g_1 \times g_2(r_i,1^{(n-2)},(a_1,a_2)) \subseteq Q_1 \times M_2$. Thus $Q_1 \times M_2$ is an $n$-ary weakly classical prime subhypermodule of $M_1 \times M_2$.
\end{proof}
Let $(M_1,f_1,g_1)$ and $(M_2,f_2,g_2)$ are two $(m,n)$-hypermodules over $(R_1,f^{\prime}_1,g^{\prime}_1)$ and $(R_2,f^{\prime}_2,g^{\prime}_2)$, respectively. Then the $(m,n)$-hypermodule $(M_1 \times M_2, f_1 \times f_2, g_1 \times g_2)$ over $(R_1 \times R_2,f^{\prime}_1 \times f^{\prime}_2, g^{\prime}_1 \times g^{\prime}_2)$ is defined by $m$-ary hyperoperation $f_1 \times f_2$ and $n$-ary external hyperoperation $g_1 \times g_2$, as follows:

$ f_1 \times f_2 ((a_1,b_1), \cdots,(a_m,b_m))=\{(x_1,x_2) \ \vert \ x_1 \in f_1(a_1^m), x_2 \in f_2(b_1^m)\}$\\
$ g_1 \times g_2((r_1,s_1),\cdots,(r_{n-1},s_{n-1}),(a,b))=\{(y_1,y_2) \ \vert \ y_1 \in g_1(r_1^{n-1},a), y_2 \in g_2(s_1^{n-1},b)\}$

for all $a_1^m,a \in M_1$, $b_1^m,b \in M_2$, $r_1^{n-1} \in R_1$ and $s_1^{n-1} \in R_2$.
\begin{theorem}
Let $(M_1 \times M_2,f_1 \times f_2, g_1 \times g_2)$ be an $(m,n)$-hypermodule over $(R_1 \times R_2,f^{\prime}_1 \times f^{\prime}_2,g^{\prime}_1 \times g^{\prime}_2)$ such that $(M_1,f_1,g_1)$ is an $(m,n)$-hypermodule over $(R_1, f^{\prime}_1,g^{\prime}_1)$ and $(M_2,f_2,g_2)$ is an $(m,n)$-hypermodule over $(R_2, f^{\prime}_2, g^{\prime}_2)$. Let $Q_1 \times M_2$ be a proper subhypermodule of $M_1 \times M_2$. Then the followings are equivalent:
\begin{itemize} 
\item[\rm(1)]~ $Q_1$ is an $n$-ary classical prime subhypermodule of $M_1$.
\item[\rm(2)]~ $Q_1 \times M_2$ is an $n$-ary classical prime subhypermodule of $M_1 \times M_2$.
\item[\rm(3)]~ $Q_1 \times M_2$ is an $n$-ary weakly classical prime subhypermodule of $M_1 \times M_2$.
\end{itemize}
\end{theorem}
\begin{proof}
$(1) \Longrightarrow (2)$ Let $g_1 \times g_2((r_1,s_1), \cdots (r_{n-1},s_{n-1}),(a,b)) =\{(y_1,y_2) \ \vert \ y_1 \in g_1(r_1^{n-1},a), y_2 \in g_2(s_1^{n-1},b)\}\subseteq Q_1 \times M_2$ for some $(r_1,s_1), \cdots , (r_{n-1},s_{n-1}) \in R_1 \times R_2$, $(a,b) \in M_1 \times M_2$. Therefore $g_1(r_1^{n-1},a) \subseteq Q_1$. Since $Q_1$ is an $n$-ary classical prime subhypermodule of $M_1$, we conclude that $g_1(r_i,1^{(n-2)},a) \subseteq Q_1$ for some $1 \leq i \leq n-1$ which implies $g_1 \times g_2((r_i,s_i),(1,1)^{n-2},(a,b)) \subseteq Q_1 \times M_2$. This shows that $Q_1 \times M_2$ is an $n$-ary classical prime subhypermodule of $M_1 \times M_2$.\\
$(2) \Longrightarrow (3)$ It is obvious.\\
$(3) \Longrightarrow (1)$ Assume that $g_1(r_1^{n-1},a) \subseteq Q_1$ for some $r_1^{n-1} \in R_1$ and $a \in M_1$. Let us pick $0 \neq b \in M_2$. Then $(0,0) \notin g_1 \times g_2((r_1,s_1),\cdots,(r_{n-1},s_{n-1}),(a,b)) \subseteq Q_1 \times M_2$. Since $Q_1 \times M_2$ is an $n$-ary weakly classical prime subhypermodule of $M_1 \times M_2$, we get $g_1 \times g_2((r_i,s_i),(1,1)^{(n-2)},(a,b)) \subseteq Q_1 \times M_2$ for some $1 \leq i \leq n-1$ which shows $g_1(r_i,1^{(n-2)},a) \subseteq Q_1$. Consequently, $Q_1$ is an $n$-ary classical prime subhypermodule of $M_1$.
\end{proof}
\section{$n$-ary $\phi$-classical prime subhypermodule}
In this section, the concept of $n$-ary $\phi$-classical prime subhypermodules of an $(m,n)$-hypermodule over $R$ is introduced. The results obtained in the theorems seem to play an important role to study $n$-ary $\phi$-classical prime subhypermodules.
\begin{definition}
Let  $\mathcal{SH}(M)$ be the set of all subhypermodules of an $(m,n)$-hypermodule $M$ over $R$  and $ \phi : \mathcal{SH}(M) \longrightarrow \mathcal{SH}(M) \cup \{\varnothing\}$ be a function. A proper subhypermodule $Q$ of $M$ is said to be an $n$-ary $\phi$-classical prime subhypermodule if $r_1^{n-1} \in R$ and $a \in M$, $g(r_1^{n-1},a) \subseteq Q-\phi(Q)$ implies that $g(r_i,1^{(n-2)},a) \subseteq Q$ for some $1 \leq i \leq n-1$.
\end{definition}
\begin{example}
Assume that $\mathbb{Z}$ is the ring of integers and  $(\mathbb{Z}, f , g)$ is the $(m, n)$-hypermodule  over $(\mathbb{Z}, h, k)$ defined in Example 3.5 of \cite{sorc2}. Let for ever subhypermodule $N$ of $\mathbb{Z}$, $S_N=\{r \in \mathbb{Z} \ \vert \ g(r,1^{(n-2)},\mathbb{Z}) \subseteq N\}$ . Consider the function $\phi :\mathcal{SH}(\mathbb{Z}) \longrightarrow \mathcal{SH}(\mathbb{Z}) \cup \{\varnothing\}$ defined by $\phi(N)=g(S_N,1^{(n-2)},N)$ for ever subhypermodule $N$ of $\mathbb{Z}$. Then the subhypermodule $g(\mathbb{Z}^{n-1},p)$ of $\mathbb{Z}$ is an $n$-ary $\phi$-classical prime subhypermodule.
\end{example}
Suppose that  $N$ is a subhypermodule of an $(m,n)$-hypermodule $M$ over $R$ and  $ \phi : \mathcal{SH}(M) \longrightarrow \mathcal{SH}(M) \cup \{\varnothing\}$ is a function. Define $\phi_N$ from $\mathcal{SH}(M/N)$ into $\mathcal{SH}(M/N) \cup \{\varnothing\}$ by $\phi_N(K/N)=f(\phi(K),N,0^{(m-2)})/N$ for all $K \in \mathcal{SH}(M)$ such that $N \subseteq K$. If $\phi_N(K)=\varnothing$, then we consider $\phi_N(K/N)=\varnothing$. 
\begin{theorem}
Let  $N \subseteq Q$ be  proper subhypermodules of an $(m,n)$-hypermodule $M$ over $R$ and  $ \phi : \mathcal{SH}(M) \longrightarrow \mathcal{SH}(M) \cup \{\varnothing\}$ be a function. If $Q$ is an $n$-ary $\phi$-classical prime subhypermodule of $M$, then $Q/N$ is a $\phi_N$-classical prime subhypermodule of $M/N$.
\end{theorem}
\begin{proof}
Let $G(r_1^{n-1},f(a,N,0^{(n-m)})) \subseteq Q/N-\phi_N(Q/N).$ Therefore $g(r_1^{n-1},a) \subseteq Q-\phi(Q)$ which implies $g(r_i, 1^{(n-2)},a) \subseteq Q$ for some $1 \leq i \leq n-1$ since $Q$ is an $n$-ary $\phi$-classical prime subhypermodule of $M$. Thus $G(r_i,1^{(n-2)},f(a,N,0^{(m-2)})) \subseteq Q/N$. This shows that $Q/N$ is a $\phi_N$-classical prime subhypermodule of $M/N$.
\end{proof}
\begin{theorem} \label{90}
Let $N$ and $Q$ be  proper subhypermodules of an $(m,n)$-hypermodule $M$ over $R$ such that $N \subseteq Q$. Suppose that $ \phi : \mathcal{SH}(M) \longrightarrow \mathcal{SH}(M) \cup \{\varnothing\}$ is a function. Then the followings hold:
\begin{itemize} 
\item[\rm(1)]~ If $Q$ is an $n$-ary $\phi$-classical prime subhypermodule of $M$ such that $\phi(Q) \subseteq N$, then $Q/N$ is an $n$-ary weakly classical prime subhypermodule of $M/N$. 
\item[\rm(2)]~ If $Q/N$ is an $n$-ary $\phi_N$-classical prime subhypermodule of $M/N$ such that $N \subseteq \phi(Q)$, then $Q$ is an $n$-ary $\phi$-classical prime subhypermodule of $M$.
\item[\rm(3)]~If $N$ is an $n$-ary $\phi$-classical prime subhypermodule of $M$ such that $\phi(N) \subseteq \phi(Q)$ and $Q/N$ is an $n$-ary weakly classical prime subhypermodule of $M/N$, then $Q$ is an $n$-ary $\phi$-classical prime subhypermodule of $M$.
\end{itemize}
\end{theorem}
\begin{proof}
$(1)$ Let $0 \notin G(r_1^{n-1},f(a,N,0^{(m-2)}) \subseteq Q/N$ for some $r_1^{n-1} \in R$ and $a \in M$. Since $\phi(Q) \subseteq N$,  we conclude that $g(r_1^{n-1},a) \subseteq Q-\phi(Q)$. Since  $Q$ is an $n$-ary $\phi$-classical prime subhypermodule of $M$, we get $g(r_i,1^{(n-2)},a) \subseteq Q$ for some $1 \leq i \leq n-1$. It gives $G(r_i,1^{(n-2)},f(a,N,0^{(m-2)})) \subseteq Q/N$. Thus $Q/N$ is an $n$-ary weakly classical prime subhypermodule of $M/N$. 

$(2)$ Let $g(r_1^{n-1},a) \subseteq Q-\phi(Q)$ for some $r_1^{n-1} \in R$ and $a \in M$. Then we conclude that $G(r_1^{n-1},f(a,N,0^{(m-2)})) \subseteq Q/N-\phi_N(Q/N)=Q/N-(\phi(Q)/N)$. Since $Q/N$ is an $n$-ary $\phi_N$-classical prime subhypermodule of $M/N$, we obtain $G(r_i,1^{(n-2)},f(a,N,0^{(m-2)})) \subseteq Q/N$ for some $1 \leq i \leq n-1$. It follows that  $g(r_i,1^{(n-2)},a) \subseteq Q$. Consequently, $Q$ is an $n$-ary $\phi$-classical prime subhypermodule of $M$.

$(3)$ Suppose that $g(r_1^{n-1},a) \subseteq Q-\phi(Q)$ for some $r_1^{n-1} \in R$ and $a \in M$. From $\phi(N) \subseteq \phi(Q)$, it follows that $g(r_1^{n-1},a) \nsubseteq \phi(N)$. Let $g(r_1^{n-1},a) \subseteq N$. Since $N$ is an $n$-ary $\phi$-classical prime subhypermodule of $M$, we get $g(r_i,1^{(n-2)},a) \subseteq N \subseteq Q$ for some $1 \leq i \leq n-1$. Now, let $g(r_1^{n-1},a)\nsubseteq N$. It implies that $0 \notin G(r_1^{n-1},f(a,N,0^{(m-2)})) \subseteq Q/N$ and so $G(r_i,1^{(n-2)},f(a,N,0^{(m-2)})) \subseteq Q/N$ for some $1 \leq i \leq n-1$ since $Q/N$ is an $n$-ary weakly classical prime subhypermodule of $M/N$. It shows that $g(r_i,1^{(n-2)},a) \subseteq Q$ for some $1 \leq i \leq n-1$. Hence $Q$ is an $n$-ary $\phi$-classical prime subhypermodule of $M$.
\end{proof}
In view of Theorem \ref{90}, the following result is obtained.
\begin{corollary}
Assume that   $Q$ is  a proper subhypermodule of an $(m,n)$-hypermodule $M$ over $R$ and $ \phi : \mathcal{SH}(M) \longrightarrow \mathcal{SH}(M) \cup \{\varnothing\}$ is a function. Then the following conditions  are equivalent:
\begin{itemize} 
\item[\rm(1)]~ $Q$ is an $n$-ary $\phi$-classical prime subhypermodule of $M$.
\item[\rm(2)]~ $Q/\phi(Q)$ is an $n$-ary weakly classical prime subhypermodule of $M/\phi(Q)$. 
\end{itemize}
\end{corollary}
\begin{theorem}
Suppose that   $Q$ is  a proper subhypermodule of an $(m,n)$-hypermodule $M$ over $R$ and $ \phi : \mathcal{SH}(M) \longrightarrow \mathcal{SH}(M) \cup \{\varnothing\}$ and  $\phi^{\prime} : \mathcal{HI}(R) \longrightarrow \mathcal{HI}(R) \cup \{\varnothing\}$  are two functions such that $\mathcal{HI}(R)$  is the set of all hyperideals of  $R$. Then the followings hold:
\begin{itemize} 
\item[\rm(1)]~ Let $Q$ be an $n$-ary $\phi$-classical prime subhypermodule of $M$. Then    $g^{\prime}(r_1^n) \in Q_a-\phi^{\prime}(Q_a)$ for $r_1^n \in R$ and for all $a \in M-Q$ with $\phi(Q)_a \subseteq \phi^{\prime}(Q_a)$ implies that $r_i \in Q_a$ for some $1 \leq i \leq n$. 
\item[\rm(2)]~ If  $g^{\prime}(r_1^n) \in Q_a-\phi(Q_a)$ for some $r_1^n \in R$ and for every $a \in M-Q$ with $\phi^{\prime}(Q_a) \subseteq \phi(Q)_a$ implies that $r_i \in Q_a$ for some $1 \leq i \leq n$, then $Q$ is an $n$-ary $\phi$-classical prime subhypermodule of $M$.
\end{itemize}
\end{theorem}
\begin{proof}
$(1)$ Let $Q$ be an $n$-ary $\phi$-classical prime subhypermodule of $M$. Pick $a \in M-Q$ with $\phi(Q)_a \subseteq \phi^{\prime}(Q_a)$. Assume that $g^{\prime}(r_1^n) \in Q_a-\phi^{\prime}(Q_a)$ for some $r_1^n \in R$. This means $g(g^{\prime}(r_1^n),1^{(n-2)},a)=g(r_1^{n-2},g^{\prime}(r_{n-1}^n,1^{(n-2)}),a) \subseteq Q-\phi(Q)$. Since $Q$ is an $n$-ary $\phi$-classical prime subhypermodule of $M$, then $g(r_i,1^{(n-2)},a) \subseteq Q$ for some $1 \leq i \leq n-2$ or $g(g^{\prime}(r_{n-1}^n,1^{(n-2)}),1^{(n-2)},a)=g(r_{n-2},r_n,1^{(n-2)},a) \subseteq Q$. In the second possibility, we have $g(r_i,1^{(n-2)},a) \subseteq Q$ for some $i \in \{n-1,n\}$. Then we conclude that $r_i \in Q_a$ for some $1 \leq i \leq n$, as needed.

$(2)$ Suppose that  $g(r_1^{n-1},a) \subseteq Q-\phi(Q)$ for some $r_1^{n-1} \in R$ and $a \in M$. Let $a \in Q$. Then  the claim follows. If $a \notin Q$.  From $g^{\prime}(r_1^{n-1},1) \in Q_a-\phi^{\prime}(Q_a)$, it follows that $r_i \in Q_a$ for some $1 \leq i \leq n-1$. Hence $g(r_i ,1^{(n-2)},a) \subseteq Q$. Consequently, $Q$ is an $n$-ary $\phi$-classical prime subhypermodule of $M$.
\end{proof}

\begin{theorem}
Let $(M_1,f_1,g_1)$ and $(M_2,f_2,g_2)$ be two $(m,n)$-hypermodules over $R$ and $h:M_1 \longrightarrow M_2$ be an epimorphism. Let $ \phi_1 : \mathcal{SH}(M_1) \longrightarrow \mathcal{SH}(M_1) \cup \{\varnothing\}$ and $ \phi_2 : \mathcal{SH}(M_2) \longrightarrow \mathcal{SH}(M_2) \cup \{\varnothing\}$ be two functions. 
\begin{itemize} 
\item[\rm(1)]~ If $Q_2$ is an $n$-ary $\phi_2$-classical prime subhypermodule of $M_2$ such that $\phi_1(h^{-1}(Q_2))=h^{-1}(\phi_2(Q_2))$, then $h^{-1} (Q_2)$ is an $n$-ary $\phi_1$-classical prime subhypermodule of $M_1$.
\item[\rm(2)]~ If $Q_1$ is an $n$-ary $\phi_1$-classical prime subhypermodule of $M_1$ such that $ Ker(h) \subseteq Q_1 $ and $\phi_2(h(Q_1))=h(\phi_1(Q_1))$, then $h(Q_1)$ is an $n$-ary $\phi_2$-classical prime subhypermodule of $M_2$.
\end{itemize} 
\end{theorem}
\begin{proof}
$(1)$ Assume that $g_1(r_1^{n-1},a_1) \subseteq h^{-1}(Q_2)-\phi_1(h^{-1}(Q_2))$ for some $r_1^{n-1} \in R$ and $a_1 \in M_1$. This means $h(g_1(r_1^{n-1},a_1)=g_2(r_1^{n-1},h(a_1)) \subseteq Q_2$. From $\phi_1(h^{-1}(Q_2))=h^{-1}(\phi_2(Q_2))$, it follows that $g_2(r_1^{n-1},h(a_1)) \nsubseteq \phi_2(Q_2)$. Since $Q_2$ is an $n$-ary $\phi_2$-classical prime subhypermodule of $M_2$ and $g_2(r_1^{n-1},h(a_1)) \subseteq Q_2-\phi_2(Q_2)$, we get $g_2(r_i,1^{(n-2)},h(a_1)) \subseteq Q_2$ for some $1 \leq i \leq i-1$. Then  $h(g_1(r_i,1^{(n-2)},a_1) \subseteq Q_2$ and so $g_1(r_i,1^{(n-2)},a_1) \subseteq h^{-1}(Q_2)$. Thus $h^{-1} (Q_2)$ is an $n$-ary $\phi_1$-classical prime subhypermodule of $M_1$.

$(2)$ Suppose that $g_2(r_1^{n-1},a_2) \subseteq h(Q_1)-\phi_2(h(Q_1))$ for some $r_1^{n-1} \in R$ and $a_2 \in M_2$. Since $h$ is an epimorphism, we have $h(a_1)=a_2$ for some $a_1 \in M_1$. Hence $h(g_1(r_1^{n-1},a_1))=g_2(r_1^{n-1},h(a_1))=g_2(r_1^{n-1},a_2) \subseteq h(Q_1)$ and so $g_1(r_1^{n-1},a_1) \subseteq Q_1$. From $\phi_2(h(Q_1))=h(\phi_1(Q_1))$, it follows that $g_1(r_1^{n-1},a_1) \subseteq Q_1-\phi(Q_1)$. Since $Q_1$ is an $n$-ary $\phi_1$-classical prime subhypermodule of $M_1$, we conclude that $g_1(r_i,1^{(n-2)},a_1) \subseteq Q_1$ for some $1 \leq i \leq n-1$. Thus we get $h(g_1(r_i,1^{(n-2)},a_1))=g_2(r_i,1^{(n-2)},h(a_1))=g_2(r_i,1^{(n-2)},a_2)\subseteq h(Q_1)$. Consequently,  $h(Q_1)$ is an $n$-ary $\phi_2$-classical prime subhypermodule of $M_2$.
\end{proof}
\begin{theorem}
Let   $Q$ be  a proper subhypermodule of an $(m,n)$-hypermodule $M$ over $R$ and $ \phi : \mathcal{SH}(M) \longrightarrow \mathcal{SH}(M) \cup \{\varnothing\}$ be a function. If $Q$ bis  an $n$-ary $\phi$-classical prime subhypermodule of $M$, then $Q_{g(r_1^{n-1},a)} \subseteq \phi(Q)_{g(r_1^{n-1},a)} \cup Q_{g(r_1,1^{(n-2)},a)} \cup \cdots \cup Q_{g(r_{n-1},1^{(n-2)},a)}$ for all $r_1^{n-1} \in R$ and $a \in M$.
\end{theorem}
\begin{proof}
Let $x \in Q_{g(r_1^{n-1},a)}$. This means that $g(x,1^{(n-2)},g(r_1^{n-1},a)) \subseteq Q$. Let $g(x,1^{(n-2)},g(r_1^{n-1},a)) \subseteq \phi(Q)$. It implies that $x \in \phi(Q)_{g(r_1^{n-1},a)}$, as needed. So we consider $g(x,1^{(n-2)},g(r_1^{n-1},a)) \nsubseteq \phi(Q)$. Since $Q$ is  an $n$-ary $\phi$-classical prime subhypermodule of $M$ and $g(x,1^{(n-2)},g(r_1^{n-1},a))=g(r_1^{n-2},g^{\prime}(r_{n-1},x,1^{(n-2)}),a) \subseteq Q-\phi(Q)$, we conclude that  $g(r_i,1^{(n-2)},a) \subseteq Q$ for some $1 \leq i \leq n-2$ or $g(g^{\prime}(r_{n-1},x,1^{(n-2)}),1^{(n-2)},a)=g(x,1^{(n-2)},g(r_{n-1},1^{(n-2)},a)) \subseteq Q$.
In the former case, we get $g(x,1^{(n-2)},g(r_i,1^{(n-1)},a)) \subseteq Q$ which means $x \in Q_{g(r_i,1^{(n-2)},a)}$ for some $1 \leq i \leq n-2$. In the second case, we obtain $x \in Q_{g(r_{n-1},1^{(n-2)},a)}$. Then the claim is proved.
\end{proof}
The following theorem offers a characterization of $n$-ary $\phi$-classical prime subhypermodules of $M$.
\begin{theorem} \label{multiplication}
Let   $Q$ be  a proper subhypermodule of an $(m,n)$-hypermodule $M$ over $R$ and $ \phi : \mathcal{SH}(M) \longrightarrow \mathcal{SH}(M) \cup \{\varnothing\}$ be a function. Then $Q$ is an $n$-ary $\phi$-classical prime subhypermodule of $M$ if and only if for every hyperideals $I_1^{n-1}$ of $R$  and $a \in M$, $g(I_1^{n-1},a) \subseteq Q-\phi(Q)$ implies that $g(I_i,1^{(n-2)},a) \subseteq Q$ for some $1 \leq i \leq n-1$.
\end{theorem}\begin{proof}
$(\Longrightarrow)$ Assume that $g(I_1^{n-1},a) \subseteq Q-\phi(Q)$ for some hyperideals $I_1^{n-1}$ of $R$ and $a \in M$ but $g(I_i,1^{(n-2)},a) \nsubseteq Q$ for all $1 \leq i \leq n-1$. Then there exists $r_i \in I_i$ for each $1 \leq i \leq n-1$ such that $g(r_i,1^{(n-2)},a) \nsubseteq Q$. Since $Q$ is an $n$-ary $\phi$-classical prime subhypermodule of $M$ and $g(r_1^{n-1},a) \subseteq Q-\phi(Q)$, we conclude that $g(r_i,1^{(n-2)},a) \subseteq Q$ for some $1 \leq i \leq n-1$ which is a contradiction.

$(\Longleftarrow)$ Suppose that $g(r_1^{n-1},a) \subseteq Q-\phi(Q)$ for some $r_1^{n-1} \in R$ and $a \in M$. Then we have $g(\langle r_1 \rangle, \cdots, \langle r_{n-1} \rangle, a) \subseteq Q$. Since $g(r_1^{n-1},a) \nsubseteq \phi(Q)$, then we conclude that $g( \langle r_1 \rangle, \cdots,  \langle r_n \rangle, a) \nsubseteq \phi(Q)$. By the hypothesis, we have $g(\langle r_i \rangle, 1^{(n-2)},a) \subseteq Q$ for some $1 \leq i \leq n-1$. Therefore we get $g(r_i,1^{(n-2)},a) \subseteq Q$ which means $Q$ is an $n$-ary $\phi$-classical prime subhypermodule of $M$.
\end{proof}
Recall from \cite{sorc2} that an $(m,n)$-hypermodule $M$ over $R$ is a multiplication $(m,n)$-hypermodule if for every subhypermodule $K$ of $M$, there exists a hyperideal $I$ of $R$ with $K=g(I,1^{(n-2)},M)$. Let $K_i$ be a subhypermodule of a multiplication  $(m,n)$-hypermodule $M$ for each $1 \leq i \leq n-1$ such that $K_i=g(I_i,1^{(n-2)},M)$ for some hyperideal $I_i$ of $R$. Then the product of $K_1,\cdots,K_n$ denoted by $g(K_1^n)$ is defined by $g(K_1^n)=g(g^{\prime}((I_1^n),1^{(n-2)},M)$. Also, we define $g(K_1^{n-1},a)=g(I_1^{n-1},a)$ and $g(K_i,M^{(n-2)},a)=g(I_i,1^{(n-2)},a)$ for each $1 \leq i \leq n-1$ and for any $a \in M$.
\begin{theorem}
Let   $Q$ be  a proper subhypermodule of a multiplication $(m,n)$-hypermodule $M$ over $R$ and $ \phi : \mathcal{SH}(M) \longrightarrow \mathcal{SH}(M) \cup \{\varnothing\}$ be a function. Then $Q$ is an $n$-ary $\phi$-classical prime subhypermodule of $M$ if and only if $g(Q_1^{n-1},a) \subseteq Q-\phi(Q)$ for some subhypermodules $Q_1^{n-1}$ of $M$ and $a \in M$ implies that $g(Q_i,M^{(n-2)},a) \subseteq Q$ for some $1 \leq i \leq n-1$.
\end{theorem}
\begin{proof}
$(\Longrightarrow)$ Assume that $g(Q_1^{n-1},a) \subseteq Q-\phi(Q)$ for some subhypermodules $Q_1^{n-1}$ of $M$ and $a \in M$. Since $M$ is a multiplication $(m,n)$-hypermodule, then there exist some  hyperideals  $I_1^{n-1}$ of $R$ with $Q_i=g(I_i,1^{(n-2)},M)$ for each $1 \leq i \leq n-1$. Therefore we have $g(Q_1^{n-1},a)=g(I_1^{n-1},a) \subseteq Q-\phi(Q)$. Since $Q$ is an $n$-ary $\phi$-classical prime subhypermodule of $M$, then    $g(I_i,1^{(n-1)},a) \subseteq Q$ for some $1 \leq i \leq n-1$ by Theorem \ref{multiplication}. This means that $g(Q_i,M^{(n-2)},a) \subseteq Q$, as needed.

$(\Longleftarrow)$ Let $g(I_1^{n-1},a) \subseteq Q-\phi(Q)$ for some hyperideals $I_1^{n-1} $ of $R$ and $a \in M$. Now, we put $Q_i=g(I_i,1^{(n-2)},M)$ for each $1 \leq i \leq n-1$. Then we have $g(Q_1^{n-1},a) \subseteq Q-\phi(Q)$ which implies $g(Q_i,M^{(n-2)},a) \subseteq Q$ for some $1 \leq i \leq n-1$. Therefore $g(I_i,1^{(n-2)},a) \subseteq Q$. Thus $Q$ is an $n$-ary $\phi$-classical prime subhypermodule of $M$ by Theorem \ref{multiplication}.
\end{proof}
\begin{theorem}
Assume that $(M_1 \times M_2,f_1 \times f_2, g_1 \times g_2)$ is an $(m,n)$-hypermodule over $(R_1 \times R_2, f^{\prime}_1 \times f^{\prime}_2,g^{\prime}_1 \times g^{\prime}_2)$ such that $(M_1,f_1,g_1)$ is an $(m,n)$-hypermodule over $(R_1, f^{\prime}_1,g^{\prime}_1)$ and $(M_2,f_2,g_2)$ is an $(m,n)$-hypermodule over $(R_2, f^{\prime}_2, g^{\prime}_2)$.  Let $\phi: \mathcal{SH}(M_1 \times M_2) \longrightarrow \mathcal{SH}(M_1 \times M_2) \cup \{\varnothing\}$ be a function. If  $Q_1$ is  an $n$-ary weakly classical prime subhypermodule of $M_1$ with $\{0\} \times M_2 \subseteq \phi(Q_1 \times M_2)$, then $Q_1 \times M_2$ is an $n$-ary $\phi$-classical prime subhypermodule of $M_1 \times M_2$.
\end{theorem}
\begin{proof}
Let  $g_1 \times g_2((r_1,s_1), \cdots (r_{n-1},s_{n-1}),(a,b)) =\{(y_1,y_2) \ \vert \ y_1 \in g_1(r_1^{n-1},a), y_2 \in g_2(s_1^{n-1},b)\}\subseteq Q_1 \times M_2-\phi(Q_1 \times M_2)$ for some $(r_1,s_1), \cdots , (r_{n-1},s_{n-1}) \in R_1 \times R_2$ and $(a,b) \in M_1 \times M_2$. Therefore $0 \notin g_1(r_1^{n-1},a) \subseteq Q_1$. Since $Q_1$ is an $n$-ary weakly classical prime subhypermodule of $M_1$, we conclude that $g_1(r_i,1^{(n-2)},a) \subseteq Q_1$ for some $1 \leq i \leq n-1$ which implies $g_1 \times g_2((r_i,s_i),(1,1)^{n-2},(a,b)) \subseteq Q_1 \times M_2$. This means that $Q_1 \times M_2$ is an $n$-ary $\phi$-classical prime subhypermodule of $M_1 \times M_2$.
\end{proof}
\begin{theorem}
Suppose that  $(M_1 \times M_2,f_1 \times f_2, g_1 \times g_2)$ is an $(m,n)$-hypermodule over $(R_1 \times R_2, f^{\prime}_1 \times f^{\prime}_2,g^{\prime}_1 \times g^{\prime}_2)$ such that $(M_1,f_1,g_1)$ is an $(m,n)$-hypermodule over $(R_1, f^{\prime}_1,g^{\prime}_1)$ and $(M_2,f_2,g_2)$ is an $(m,n)$-hypermodule over $(R_2, f^{\prime}_2, g^{\prime}_2)$.  Let $\phi_1: \mathcal{SH}(M_1) \longrightarrow \mathcal{SH}(M_1 ) \cup \{\varnothing\}$ and $\phi_2: \mathcal{SH}(M_2) \longrightarrow \mathcal{SH}(M_2 ) \cup \{\varnothing\}$ be two functions such that $\phi_2(M_2)=M_2$. Then  $Q_1 \times M_2$ is an $n$-ary $\phi_1 \times \phi_2$-classical prime subhypermodule of $M_1 \times M_2$ if and only if $Q_1$ is an $n$-ary $\phi_1$-classical prime subhypermodule of $M_1$.
\end{theorem}
\begin{proof}
$(\Longrightarrow)$ Assume that $Q_1 \times M_2$ is an $n$-ary $\phi_1 \times \phi_2$-classical prime subhypermodule of $M_1 \times M_2$. Let $g_1(r_1^{n-1},a_1) \subseteq Q_1-\phi_1(Q_1)$ for some $r_1^{n-1} \in R$ and $a_1 \in M_1$. Therefore we have $g_1 \times g_2((r_1,1),\cdots,(r_{n-1},1)(a_1,a_2)) \subseteq Q_1 \times M_2-\phi_1 \times \phi_2(Q_1 \times M_2)=Q_1 \times M_2-(\phi_1(Q_1) \times \phi_2( M_2))$ for all $a_2 \in M_2$. Since  $Q_1 \times M_2$ is an $n$-ary $\phi_1 \times \phi_2$-classical prime subhypermodule of $M_1 \times M_2$, we obtain $g_1 \times g_2((r_i,1),(1,1)^{(n-2)},(a_1,a_2)) \subseteq Q_1 \times M_2$ for some $1 \leq i \leq n-1$ which means $g_1(r_i,1^{(n-2)},a_1) \subseteq Q_1$. This shows that $Q_1$ is an $n$-ary $\phi_1$-classical prime subhypermodule of $M_1$.

$(\Longleftarrow)$ Let $Q_1$ be an $n$-ary $\phi_1$-classical prime subhypermodule of $M_1$. Assume that $g_1 \times g_2((r_1,s_1), \cdots, (r_{n-1},s_{n-1})(a_1,a_2)) \subseteq Q_1 \times M_2 -\phi_1 \times \phi_2(Q_1 \times M_2)$. From $\phi_2(M_2)=M_2$, it follows that $g_1(r_1^{n-1},a_1) \subseteq Q_1-\phi_1(Q_1)$. Then we have $g_1(r_i,1^{(n-2)},a_1) \subseteq Q_1$ for some $1 \leq i \leq n-1$. So we conclude that $g_1 \times g_2((r_i,s_i), (1,1)^{(n-2)},(a_1,a_2)) \subseteq Q_1 \times M_2$. Consequently, $Q_1 \times M_2$ is an $n$-ary $\phi_1 \times \phi_2$-classical prime subhypermodule of $M_1 \times M_2$.
\end{proof}
\begin{theorem}
Let $(M_1 \times M_2,f_1 \times f_2, g_1 \times g_2)$ be an $(m,n)$-hypermodule over $(R_1 \times R_2, f^{\prime}_1 \times f^{\prime}_2,g^{\prime}_1 \times g^{\prime}_2)$ such that $(M_1,f_1,g_1)$ is an $(m,n)$-hypermodule over $(R_1, f^{\prime}_1,g^{\prime}_1)$ and $(M_2,f_2,g_2)$ is an $(m,n)$-hypermodule over $(R_2, f^{\prime}_2, g^{\prime}_2)$.  Assume that  $\phi_1: \mathcal{SH}(M_1) \longrightarrow \mathcal{SH}(M_1 ) \cup \{\varnothing\}$ and $\phi_2: \mathcal{SH}(M_2) \longrightarrow \mathcal{SH}(M_2 ) \cup \{\varnothing\}$ be two functions. If  $Q_1 \times Q_2$ is an $n$-ary $\phi_1 \times \phi_2$-classical prime subhypermodule of $M_1 \times M_2$, then  $Q_1$ is an $n$-ary $\phi_1$-classical prime subhypermodule of $M_1$ and $Q_2$ is an $n$-ary $\phi_2$-classical prime subhypermodule of $M_2.$
\end{theorem}
\begin{proof}
Let $Q_1 \times Q_2$ be an $n$-ary $\phi_1 \times \phi_2$-classical prime subhypermodule of $M_1 \times M_2$. Assume that $g_1(r_1^{n-1},a_1) \subseteq Q_1 - \phi_1(Q_1)$ for some $r_1^{n-1} \in R$ and $a \in M_1$. Pick $a_2 \in Q_2$. So $g_1 \times g_2((r_1,1),\cdots,(r_{n-1},1)(a_1,a_2)) \subseteq Q_1 \times Q_2-\phi_1 \times \phi_2(Q_1 \times Q_2)$. Therefore we get $g_1 \times g_2((r_i,1),(1,1)^{(n-2)},(a_1,a_2)) \subseteq Q_1 \times Q_2$ for some $1 \leq i \leq n-1$ which implies $g_1(r_i,1^{(n-2)},a_1) \subseteq Q_1$. Thus $Q_1$ is an $n$-ary $\phi_1$-classical prime subhypermodule of $M_1$. Similarly, we can show that $Q_2$ is an $n$-ary $\phi_2$-classical prime subhypermodule of $M_2.$
\end{proof}
\section{conclusion}
The notion of prime submodules has a significant place in the theory of modules, and it is used to characterize certain classes of modules. In this paper, we studied some generalizations on this issue  in the context of $(m, n)$-hypermodules. We introduced $n$-ary classical prime, $n$-ary weakly classical prime and $n$-ary $\phi$-classical prime subhypermodules.
In this direction we  gave some
characterizations of such subhypermodules. The future work can be on
defining the notions of  classical primary, weakly classical primary and  $\phi$-classical primary subhypermodules of an $(m,n)$-hypermodules over a Krasner $(m,n)$-hyperring.

\end{document}